\documentclass[12pt,bezier]{article}
\usepackage{times}
\usepackage{booktabs}
\usepackage{pifont}
\usepackage{floatrow}
\usepackage{caption}
\usepackage{mathrsfs}
\usepackage[fleqn]{amsmath}
\usepackage{amsfonts,amsthm,amssymb,mathrsfs,bbding}
\usepackage{txfonts}
\usepackage{graphics,multicol}
\usepackage{graphicx}
\usepackage{color}
\usepackage{amssymb}
\usepackage{caption}
\usepackage{cite}
\usepackage{latexsym,bm}
\usepackage{indentfirst}
\usepackage{color}
\usepackage[colorlinks=true,anchorcolor=blue,filecolor=blue,linkcolor=blue,urlcolor=blue,citecolor=blue]{hyperref}
\usepackage{extarrows}
\usepackage{cite}
\usepackage{latexsym,bm}
\usepackage{mathtools}
\evensidemargin=\oddsidemargin\topmargin=-1.5cm

\newtheorem{thm}{Theorem}[section]

\newtheorem{claim}{Claim}
\newtheorem{rem}{Remark}[section]

\newtheorem{lem}{Lemma}[section]

\newtheorem{pro}{Proposition}[section]

\theoremstyle{definition}

\addtocounter{section}{0}

\begin{document}
\title{An improvement of sufficient condition for $k$-leaf-connected graphs\footnote{Supported by National Natural Science Foundation of China 
(Nos. 11971445 and 12271439),
Natural Science Foundation of Henan Province (No. 202300410377) and
Research Program of Science and Technology at Universities of Inner Mongolia Autonomous Region (No. NJZY22280).}}
\author{{\bf Tingyan Ma$^{a}$}, {\bf Guoyan Ao$^{b, c}$}\thanks{Corresponding author. E-mail addresses: matingylw@163.com, aoguoyan@163.com,
rfliu@zzu.edu.cn, lgwangmath@163.com, yhumath@163.com.},
{\bf Ruifang Liu$^{b}$}, {\bf Ligong Wang$^{a}$}, {\bf Yang Hu$^{a}$} \\
{\footnotesize $^a$ School of Mathematics and Statistics, Northwestern Polytechnical University, Xi'an, Shaanxi 710129, China} \\
{\footnotesize $^b$ School of Mathematics and Statistics, Zhengzhou University, Zhengzhou, Henan 450001, China} \\
{\footnotesize $^c$ School of Mathematics and Statistics, Hulunbuir University, Hailar, Inner Mongolia 021008, China}}
\date{}

\date{}
\maketitle
{\flushleft\large\bf Abstract}
For integer $k\geq2,$ a graph $G$ is called $k$-leaf-connected  if $|V(G)|\geq k+1$ and given any subset $S\subseteq V(G)$ with $|S|=k,$ $G$ always has a spanning tree $T$ such that $S$ is precisely the set of leaves of $T.$ Thus a graph is $2$-leaf-connected if and only if it is Hamilton-connected.
In this paper, we present a best possible condition based upon the size to guarantee a graph to be $k$-leaf-connected, which not only improves the results
of Gurgel and Wakabayashi [On $k$-leaf-connected graphs, J. Combin. Theory Ser. B 41 (1986) 1-16] and Ao, Liu, Yuan and Li [Improved sufficient conditions for $k$-leaf-connected graphs, Discrete Appl. Math. 314 (2022) 17-30], but also extends the result of Xu, Zhai and Wang [An improvement of spectral conditions for Hamilton-connected graphs, Linear Multilinear Algebra, 2021].
Our key approach is showing that an $(n+k-1)$-closed non-$k$-leaf-connected graph must contain a large clique
if its size is large enough. As applications, sufficient conditions for a graph to be $k$-leaf-connected in terms of the (signless Laplacian) spectral radius of $G$ or its complement are also presented.

\begin{flushleft}

\end{flushleft}

\section{Introduction}
In this paper, we consider simple, undirected and connected graphs. Let $G$ be a graph with vertex set $V(G)=\{v_{1}, v_{2}, \ldots, v_{n}\}$ and edge set $E(G)$.
The order and size of $G$ are denoted by $|V(G)|=n$ and $|E(G)|=e(G)$, respectively.
For any vertex $u \in V(G)$, we denote by $d_{G}(u)$ the degree of vertex $u$ in $G$ and
by $(d_{1}, d_{2}, \ldots, d_{n})$ the degree sequence of $G$ with $d_{1}\leq d_{2}\leq \cdots \leq d_{n}$.
Let $G_{1}$ and $G_{2}$ be two vertex-disjoint graphs. We denote by $G_{1}+G_{2}$ the disjoint union of $G_{1}$ and $G_{2}.$
The join $G_{1}\vee G_{2}$ is the graph obtained from $G_{1}+G_{2}$ by adding all possible edges between $V(G_1)$ and $V(G_2)$.
We denote by $\delta$, $\overline{G}$, $\omega(G)$ the minimum degree,
the complement and the clique number of $G,$ respectively. For undefined terms and notions one can refer to
\cite{Bondy2008} and \cite{Brouwer2011}.

Let $A(G)$ be the adjacency matrix and $D(G)$ be the diagonal degree matrix of $G$. Let $Q(G)=D(G)+A(G)$ be the signless Laplacian matrix of $G$.
The largest eigenvalues of $A(G)$ and $Q(G)$, denoted by $\rho(G)$ and $q(G)$, are called the spectral radius and the signless Laplacian spectral radius of $G$, respectively.

The concept of closure of a graph was used implicitly by Ore \cite{Ore1960},
and formally introduced by Bondy and Chvatal \cite{Bondy1976}.
Fix an integer $l\geq 0$, the $l$-closure of a graph $G$ is the graph obtained from $G$
by successively joining pairs of nonadjacent vertices whose degree sum is at least $l$
until no such pair exists. Denote by $C_{l}(G)$ the $l$-closure of $G.$ Then we have
$$d_{C_{l}(G)}(u)+d_{C_{l}(G)}(v)\leq l-1$$ for every pair of nonadjacent vertices $u$ and $v$ of $C_{l}(G).$

For integer $k\geq2,$ a graph $G$ is called $k$-leaf-connected if $|V(G)|\geq k+1$ and given any subset $S\subseteq V(G)$ with $|S|=k,$
$G$ always has a spanning tree $T$ such that $S$ is precisely the set of leaves of $T.$
Thus a graph is $2$-leaf-connected if and only if
it is Hamilton-connected. Hence $k$-leaf-connectedness of a graph is a natural generalization of Hamilton-connectedness.
Gurgel and Wakabayashi \cite{Gurgel1986} proved that if $G$ is a $k$-leaf-connected graph of order $n$, where $2\leq k \leq n-2$,
then $G$ is $(k+1)$-connected. Hence $\delta\geq k+1$ is a trivial necessary condition for a graph to be $k$-leaf-connected.

Determining whether a given graph is $k$-leaf-connected is NP-complete. Gurgel and Wakabayashi \cite{Gurgel1986} initially proved the following sufficient condition in terms of $e(G)$ to guarantee a graph $G$ to be $k$-leaf-connected.

\begin{thm}[Gurgel and Wakabayashi \cite{Gurgel1986}]\label{th1}
Let $G$ be a connected graph of order $n$ with minimum degree $\delta \geq k+1,$ where $2\leq k\leq n-4.$
If $$e(G)\geq {n-1\choose2}+k+1,$$ then $G$ is $k$-leaf-connected.
\end{thm}

Ao, Liu, Yuan and Li \cite{Ao2022} presented the following sufficient condition for a graph to be $k$-leaf-connected
and improved the result of Theorem \ref{th1}.

\begin{thm}[Ao, Liu, Yuan and Li \cite{Ao2022}]\label{th2}
Let $G$ be a connected graph of order $n$ and minimum degree $\delta\geq k+1$,
where $2\leq k\leq n-4$. If
$$e(G)\geq {n-2\choose 2}+2k+2,$$
then $G$ is $k$-leaf-connected unless $G\in \{K_{3}\vee(K_{n-5}+2K_{1}), K_{4}\vee (K_{2}+3K_{1}), K_{6}\vee 6K_{1}, K_{5}\vee 5K_{1},
K_{4}\vee (K_{1,4}+K_{1}), K_{3}\vee K_{2,5}, K_{4}\vee 4K_{1},
K_{3}\vee (K_{1,3}+K_{1}), K_{2}\vee K_{2,4}\}.$
\end{thm}

As a special case of $k$-leaf-connectedness, there are many sufficient conditions to assure a graph to be $2$-leaf-connected
(see for example \cite{Wei2019, Yu2013, Zhou2017, Zhou}). By introducing the minimum degree $\delta$ as a new parameter,
Chen and Zhang \cite{Chen2018} presented a sufficient condition for a graph with $\delta\geq t\geq 2$ to be Hamilton-connected:
$e(G)\geq {n-t+1\choose 2}-\frac{t^{2}-3t-2}{2}$.
Zhou and Wang \cite{Zhou2020} proved a better condition for a graph to be Hamilton-connected: $e(G)\geq {n-t\choose 2}+t^{2}+t$.
Recently, Xu, Zhai and Wang \cite{Xu2021} improved the results of \cite{Chen2018} and \cite{Zhou2020}. Define
$L^{t}_{n}=K_{2}\vee (K_{n-t-1}+K_{t-1})~~(2\leq t\leq \frac{n}{2}), N^{t}_{n}=K_{t}\vee (K_{n-2t+1}+(t-1)K_{1})~~(2\leq t\leq \frac{n}{2})$,
and $M^{t}_{n}=K_{t+1}\vee (K_{n-2t-1}+tK_{1})~~(2\leq t\leq \frac{n-1}{2}).$

\begin{thm}[Xu, Zhai and Wang \cite{Xu2021}]\label{th3}
Let $G$ be a connected graph of order $n\geq 6t+3$ with $\delta\geq t\geq 2$. If
$$e(G)\geq {n-t\choose 2}+t^{2}+2,$$
then $G$ is Hamilton-connected unless $C_{n+1}(G)\in \{L^{t}_{n}, N^{t}_{n}, M^{t}_{n}\}.$
\end{thm}

Inspired by the ideas from the conjecture by Erd\H{o}s and Hajnal \cite{Erd1989} and the result on Hamilton-connected graphs by Xu, Zhai and Wang \cite{Xu2021}, we first show that an $(n+k-1)$-closed non-$k$-leaf-connected graph $G$ must contain a large clique
if its number of edges is large enough. Using the key approach and typical spectral techniques,
we present a best possible condition based upon the size to guarantee a graph to be $k$-leaf-connected as follows.
Our main result not only improves the result of Theorem \ref{th2},
but also extends the result on Hamilton-connected graphs in Theorem \ref{th3}.

\begin{thm}\label{main}
Let $G$ be a connected graph of order $n\geq k+17$ and minimum degree $\delta\geq k+1$, where $k\geq 2$.
If
$$e(G)\geq {n-3\choose 2}+3k+5,$$
then $G$ is $k$-leaf-connected unless $C_{n+k-1}(G)\in\{ K_{k}\vee (K_{n-k-2}+K_{2}), K_{3}\vee (K_{n-5}+2K_{1}),
K_{4}\vee (K_{n-7}+3K_{1})\}$.
\end{thm}

\section{Preliminaries}
We will present in this section some important results that will be used in our subsequent arguments.
Gurgel and Wakabayashi \cite{Gurgel1986} proved a sufficient condition in terms of the degree sequence
for a graph to be $k$-leaf-connected.

\begin{lem}[Gurgel and Wakabayashi \cite{Gurgel1986}]\label{le16}
Let $k$ and $n$ be such that $2\leq k\leq n-3.$ Let $G$ be a graph with degree sequence $d_{1}\leq d_{2}\leq \cdots \leq d_{n}.$
Suppose that there is no integer $i$ with $k\leq i\leq\frac{n+k-2}{2}$ such that
$d_{i-k+1}\leq i$ and $d_{n-i}\leq n-i+k-2.$ Then $G$ is $k$-leaf-connected.
\end{lem}

\begin{lem}[Gurgel and Wakabayashi \cite{Gurgel1986}]\label{le111}
Let $G$ be a graph and $k$ be an integer with $2\leq k \leq n-1.$
Then $G$ is $k$-leaf-connected
if and only if the $(n+k-1)$-closure $C_{n+k-1}(G)$ of $G$ is $k$-leaf-connected.
\end{lem}

An important upper bound on the spectral radius $\rho(G)$ is as follows.

\begin{lem}[Hong, Shu and Fang \cite{Hong2001}, Nikiforov \cite{Nikiforov2002}]\label{le8}
Let $G$ be a graph with minimum degree $\delta.$ Then
$$\rho(G)\leq \frac{\delta-1}{2}+\sqrt{2e(G)-\delta n+\frac{(\delta+1)^{2}}{4}}.$$
\end{lem}

The following observation is very useful when we use the above upper bound on $\rho(G).$

\begin{pro}[Hong, Shu and Fang \cite{Hong2001}, Nikiforov \cite{Nikiforov2002}]\label{pro9}
For graph $G$ with $2e(G)\leq n(n-1),$ the function
$$f(x)=\frac{x-1}{2}+\sqrt{2e(G)-nx+\frac{(x+1)^{2}}{4}}$$
is decreasing with respect to $x$ for $0\leq x\leq n-1.$
\end{pro}

Feng and Yu \cite{Feng2009} proved an upper bound on $q(G)$, which has been widely used in the literature.

\begin{lem}[Feng and Yu \cite{Feng2009}]\label{le10}
Let $G$ be a connected graph on $n$ vertices and $e(G)$ edges.
Then
$$q(G)\leq \frac{2e(G)}{n-1}+n-2.$$
\end{lem}

Let $M$ be the following $n\times n$ matrix
\[
M=\left(\begin{array}{ccccccc}
M_{1,1}&M_{1,2}&\cdots &M_{1,m}\\
M_{2,1}&M_{2,2}&\cdots &M_{2,m}\\
\vdots& \vdots& \ddots& \vdots\\
M_{m,1}&M_{m,2}&\cdots &M_{m,m}\\
\end{array}\right),
\]
whose rows and columns are partitioned into subsets $X_{1}, X_{2},\ldots ,X_{m}$ of $\{1,2,\ldots, n\}$.
The quotient matrix $R(M)$ of the matrix $M$ (with respect to the given partition)
is the $m\times m$ matrix whose entries are the
average row sums of the blocks $M_{i,j}$ of $M$.
The above partition is called equitable
if each block $M_{i,j}$ of $M$ has constant row (and column) sum.

\begin{lem}[Brouwer and Haemers \cite{Brouwer2011}, Godsil and Royle \cite{Godsil2001}, Haemers \cite{Haemers1995}]\label{le15}
Let $M$ be a real symmetric matrix and let $R(M)$ be its equitable quotient matrix.
Then the eigenvalues of the quotient matrix $R(M)$ are eigenvalues of $M$.
Furthermore, if $M$ is nonnegative and irreducible, then the spectral radius of the quotient matrix $R(M)$ equals to the spectral
radius of $M$.
\end{lem}

\section{Proof of Theorem \ref{main}}
Before presenting our main result, we first show that an $(n+k-1)$-closed non-$k$-leaf-connected graph $G$ must contain a large clique
if its number of edges is large enough. We denote by $\omega(G)$ the clique number of $G.$ Let $(d_{1}, d_{2}, \ldots, d_{n})$ be the degree sequence of $G$, where $d_{1}\leq d_{2}\leq \cdots \leq d_{n}$.

\begin{lem}\label{le11}
Let $G$ be an $(n+k-1)$-closed non-$k$-leaf-connected graph of order $n\geq k+17$
with $\delta\geq k+1$ and $k\geq2$. If $$e(G)\geq {n-3\choose 2}+3k+5,$$ then $\omega(G)=n-2$ unless $G\cong K_{4}\vee(K_{n-7}+3K_{1})$.
\end{lem}

\begin{proof} Note that $\delta\geq k+1.$ First we claim that $\omega(G)\leq n-2$.
Otherwise, suppose that $\omega(G)\geq n-1$, then $G$ contains an $(n-1)$-clique,
and hence for any two vertices $u,v\in V(G),$ we always have $d_{G}(u)+d_{G}(v)\geq n+k-1.$
If there exists two vertices $uv\notin E(G),$ then $d_{G}(u)+d_{G}(v)\leq n+k-2$ since $G$ is an $(n+k-1)$-closed graph, a contradiction.
Hence any two vertices of $G$ are adjacent. That is, $G\cong K_{n},$ and obviously $G$ is $k$-leaf-connected, a contradiction.

Let $(d_{1}, d_{2}, \ldots ,d_{n})$ be the degree sequence of $G$
with $d_{1}\leq d_{2}\leq \cdots \leq d_{n}.$ Note that $G$ is not $k$-leaf-connected.
By Lemma \ref{le16}, there exists an integer $i$ with
$k\leq i\leq\frac{n+k-2}{2}$ such that $d_{i-k+1}\leq i$ and $d_{n-i}\leq n-i+k-2$.
Then we have
\begin{eqnarray*}
e(G)&=&\frac{1}{2}\sum_{j=1}^{n}d_{j} \\
&=&\frac{1}{2}(\sum_{j=1}^{i-k+1}d_{j}+\sum_{j=i-k+2}^{n-i}d_{j}+\sum_{j=n-i+1}^{n}d_{j}) \\
&\leq&\frac{1}{2}[(i-k+1)i+(n-2i+k-1)(n-i+k-2)+i(n-1)] \\
&=&{n-3\choose 2}+3k+5+\frac{f_{1}(i)}{2},
\end{eqnarray*}
where $$f_{1}(i)=3i^{2}-(2n+4k-5)i+(2k+4)n+k^{2}-9k-20.$$
By the assumption $e(G)\geq {n-3\choose 2}+3k+5,$ then we have $f_{1}(i)\geq0$.
Note that $k+1\leq \delta \leq d_{i-k+1}\leq i\leq\frac{n+k-2}{2}$. We shall divide the proof into the
following three cases.

\vspace{1.5mm}
\noindent\textbf{Case 1.} $k+3 \leq i\leq\frac{n+k-2}{2}.$
\vspace{1mm}

Since $f_{1}''(i)=6>0$, then $f_{1}(i)$ is a concave function on $i$. For $n\geq k+17$, we have
$$f_{1}(k+3)=-2n+2k+22<0,$$
$$\mbox{and}~~~~~f_{1}(\frac{n+k-2}{2})=-\frac{n^{2}}{4}+\frac{k+11}{2}n-\frac{k^{2}}{4}-\frac{11k}{2}-22<0.$$
This implies that $f_{1}(i)<0$, a contradiction.

\vspace{1.5mm}
\noindent\textbf{Case 2.} $i=k+2.$
\vspace{1mm}

Then the corresponding degree sequence of $G$ is
$$\underbrace{d_{1}\leq d_{2}\leq d_{3}\leq k+2}_{V_{1}},~\underbrace{d_{4}\leq d_{5}\leq \cdots\leq d_{n-k-2}\leq n-4}_{V_{2}},
~\underbrace{d_{n-k-1}\leq d_{n-k}\leq \cdots \leq d_{n}\leq n-1}_{V_{3}}.$$
According to the above degree sequence, we divide $V(G)$ into three parts: $V_{1}$, $V_{2}$ and $V_{3}.$

\begin{claim}\label{cla1}
{\rm There is no vertex of degree less than $\frac{n+k-1}{2}$ in $V_{2}$.}
\end{claim}

\begin{proof}
Suppose that there exists a vertex of degree less than  $\frac{n+k-1}{2}$ in $V_{2}$. Then
\begin{eqnarray*}
e(G)&=&\frac{1}{2}\sum_{j=1}^{n}d_{j} \\
&<&\frac{1}{2}\big[3(k+2)+(n-k-6)(n-4)+(k+2)(n-1)+\frac{n+k-1}{2}\big] \\
&=&{n-3\choose 2}+3k+5-\frac{n-k-11}{4}\\
&\leq&{n-3\choose 2}+3k+5-\frac{3}{2}\\
&<&e(G),
\end{eqnarray*}
a contradiction, since $n\geq k+17$.
\end{proof}

By Claim \ref{cla1}, it follows that $d_{G}(u)+d_{G}(v)\geq n+k-1$ for any two different vertices $u,v\in V_{2}\cup  V_{3}.$
Note that $G$ is $(n+k-1)$-closed. Then $V_{2}\cup V_{3}$ is a clique of $G,$
and hence $$\omega(G)\geq |V_{2}\cup V_{3}|\geq (n-k-5)+(k+2)=n-3.$$ Recall that $\omega(G)\leq n-2.$
Then we have $$n-3\leq\omega(G)\leq n-2.$$
If $\omega(G)=n-2,$ then $d_{3}\geq n-3.$ Note that $d_{3}\leq k+2.$ Then $n\leq k+5,$
which contradicts $n\geq k+17$.
Thus, we have $\omega(G)=n-3$. Let $C=V_{2}\cup  V_{3}$. Note that $|C|=n-3.$ Then $C$ is a maximum clique of $G,$
and $V(G)=V_{1}\cup C.$
Notice that $k+1\leq\delta\leq d_{G}(v)\leq k+2$ for each $v\in V_{1}.$
Let $V_{1}=\{v_{1}, v_{2}, v_{3}\}$ and $V_{1}^{*}=\{v_{i}\in V_{1}~|~ d_{G}(v_{i})=k+2\}$.

\begin{figure}
\centering
\includegraphics[width=0.25\textwidth]{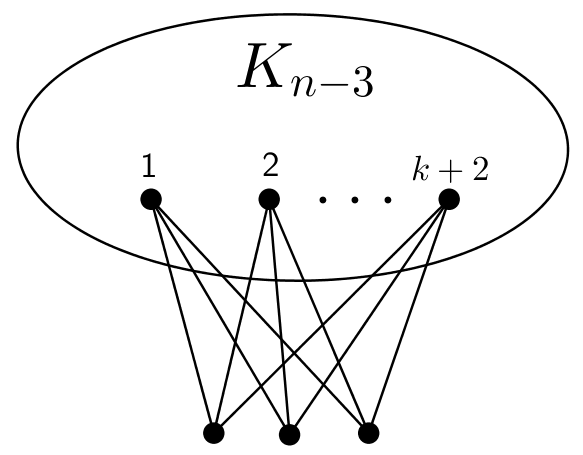}\\
\caption{Graph $K_{k+2}\vee(K_{n-k-5}+3K_{1}).$
}\label{fig1}
\end{figure}

\begin{claim}\label{cla2}
{\rm $|V_{1}^{*}|\geq 2$.}
\end{claim}

\begin{proof}
Suppose, to the contrary, that $|V_{1}^{*}|\leq 1.$ Note that $k+1\leq d_{G}(v_{i})\leq k+2$ for any $v_{i}\in V_{1}$. Then
$$e(G)\leq e(C)+\sum_{i=1}^{3}d_{G}(v_{i})\leq{n-3\choose 2}+2(k+1)+(k+2)={n-3\choose 2}+3k+4<e(G),$$
a contradiction.
\end{proof}

Define $C^{*}=\{v\in C~|~N_{G}(v)\cap V_{1}\neq \emptyset\}$.

\begin{claim}\label{cla3}
{\rm $|C^{*}|=k+2$.}
\end{claim}

\begin{proof}
By the definition of $C^{*},$ we know that $d_{G}(v)\geq n-3$ for each $v\in C^{*}$. Then
$d_{G}(v)+d_{G}(v_{i})\geq (n-3)+(k+2)=n+k-1$ for any $v\in C^{*}$ and $v_{i}\in V_{1}^{*}.$
Note that $G$ is $(n+k-1)$-closed. It follows that each vertex of $C^{*}$ is adjacent to each vertex of $V_{1}^{*}.$
Combining Claim \ref{cla2}, we have $d_{G}(v)\geq d_{C}(v)+|V_{1}^{*}|\geq (n-4)+2=n-2$ for each $v\in C^{*}$.
Therefore, $d_{G}(v)+d_{G}(v_{i})\geq (n-2)+(k+1)=n+k-1$ for any $v\in C^{*}$ and $v_{i}\in V_{1}$.
Then each vertex of $V_{1}$ is adjacent to each vertex of $C^{*},$
which implies that $|C^{*}|\leq d_{G}(v_{i})\leq k+2,$ where $v_{i}\in V_{1}$.

On the other hand, let $e(V_{1},C)$ denote the number of edges between $V_{1}$ and $C$. Notice that $e(V_{1}, C)=e(V_{1}, C^{*})=|V_{1}||C^{*}|=3|C^{*}|$ and
$e(V_{1})=\frac{1}{2}(\sum_{v_{i}\in V_{1}}d_{G}(v_{i})-3|C^{*}|)\leq \frac{3(k+2-|C^{*}|)}{2}$.
Then $$e(G)=e(C)+e(V_{1}, C^{*})+e(V_{1})\leq {n-3\choose 2}+\frac{3(k+2+|C^{*}|)}{2}.$$
Combining the assumption $e(G)\geq {n-3\choose 2}+3k+5$, we have $|C^{*}|\geq k+2$.
Therefore, $|C^{*}|=k+2.$
\end{proof}

Recall that $d_{G}(v_{i})\leq k+2$ for each $v_{i}\in V_{1}$. According to Claim \ref{cla3}, $V_{1}$ is an independent set.
This implies that $G\cong K_{k+2}\vee(K_{n-k-5}+3K_{1})$ (see Fig. \ref{fig1}).
Define
$$L=V(K_{k+2}),~~
M=V(K_{n-k-5})~~
\mbox{and} ~~N=V(3K_{1}).$$
Notice that the vertices of $N$ are only adjacent to those of $L.$
When $k\geq3$, for any $S\subseteq V(G)$ with $|S|=k,$
we always find a spanning tree $T$ (see Fig. \ref{fig2}) such that $S$ is precisely the set of leaves (labeled by red vertices) of $T$.
Hence $K_{k+2}\vee(K_{n-k-5}+3K_{1})$ is $k$-leaf-connected, which contradicts the assumption.
However, $K_{4}\vee(K_{n-7}+3K_{1})$ is not $2$-leaf-connected.
Therefore, $G\cong K_{4}\vee(K_{n-7}+3K_{1}).$

\begin{figure}
\centering
\includegraphics[width=0.9\textwidth]{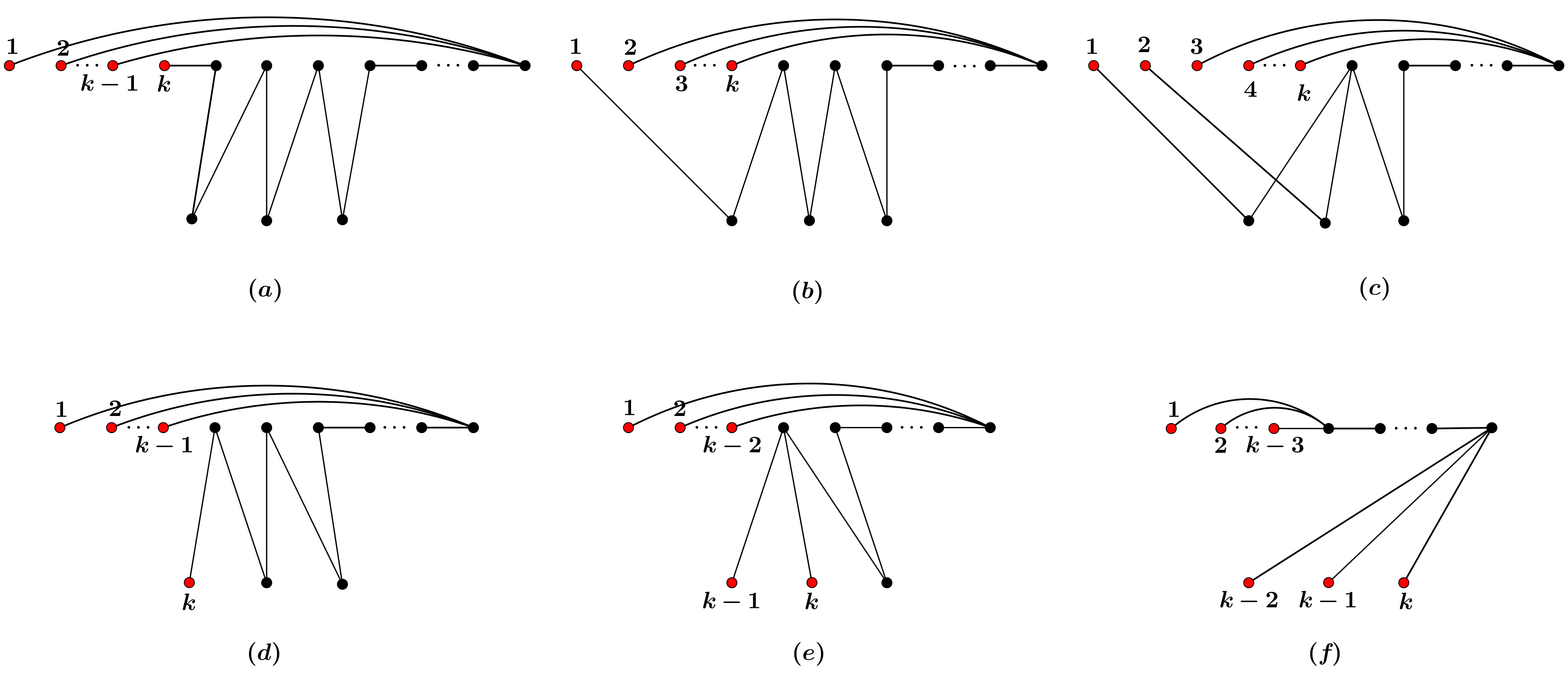}\\
\caption{
$(a)$. $k$ vertices are chosen from $M$;
$(b)$. One of $k$ vertices belongs to $L$, and the rest belong to $M$;
$(c)$. At least two vertices come from $L$, and the rest come from $M$;
$(d)$. One vertex is from $N$, and the remaining vertices come from $L\cup M$;
$(e)$. Two vertices belong to $N$, and the remaining vertices come from $L\cup M.$
$(f)$. Three vertices belong to $N$, and the remaining vertices come from $L\cup M.$
}\label{fig2}
\end{figure}

\vspace{1.5mm}
\noindent\textbf{Case 3.} $i=k+1.$
\vspace{1mm}

Then the degree sequence of $G$ is given by
$$\underbrace{d_{1}=d_{2}=k+1}_{V_{1}},~ \underbrace{d_{3}\leq d_{4}\leq \cdots\leq d_{n-k-1}\leq n-3}_{V_{2}},
~\underbrace{d_{n-k}\leq d_{n-k+1}\leq \cdots \leq d_{n}\leq n-1}_{V_{3}}.$$

\begin{claim}\label{cla4}
{\rm There are at most three vertices of degree less than $\frac{n+k-1}{2}$ in $V_{2}$.}
\end{claim}

\begin{proof}
Assume that there exist four vertices of degree less than $\frac{n+k-1}{2}$ in $V_{2}$. Then we have
\begin{eqnarray*}
e(G)&=&\frac{1}{2}\sum_{j=1}^{n}d_{j} \\
&<&\frac{1}{2}[2(k+1)+(n-k-7)(n-3)+(k+1)(n-1)+4\cdot\frac{n+k-1}{2}] \\
&=&{n-3\choose 2}+3k+4,\\
&<&e(G),
\end{eqnarray*}
a contradiction.
\end{proof}

Let $V_{2}^{*}=\{v\in V_{2}~|~d_{G}(v)\geq\frac{n+k-1}{2}\}$. By Claim \ref{cla4}, we have $|V_{2}^{*}|\geq|V_{2}|-3= n-k-6>0$.
It is clear that $d_{G}(u)+d_{G}(v)\geq n+k-1$ for any $u,v\in V_{2}^{*}\cup V_{3}.$ Note that $G$ is an $(n+k-1)$-closed graph.
This implies that $V_{2}^{*}\cup V_{3}$ is a clique of $G,$
and hence $\omega(G)\geq |V_{2}^{*}\cup V_{3}|\geq (n-k-6)+(k+1)=n-5$. Note that $\omega(G)\leq n-2$. Then we have $$n-5\leq\omega(G)\leq n-2.$$

Define $C=V_{2}^{*}\cup V_{3}$.
\begin{claim}\label{cla5}
{\rm $C$ is a maximum clique of $G$.}
\end{claim}

\begin{proof}
By the definition of $V_{2}^{*},$ we know that $d_{G}(u)<\frac{n+k-1}{2}\leq n-9<n-5$ for any $u\in V_{1}\cup (V_{2}\backslash V_{2}^{*})$, since $n\geq k+17.$
Hence there exists at least one vertex $v\in C$ such that $uv\notin E(G)$ for any $u\in V_{1}\cup (V_{2}\backslash V_{2}^{*}),$
and thus $u\notin C$. This implies that $C$ is a maximum clique of $G$.
\end{proof}

Next let $\omega(G)=\omega$ for short.

\begin{claim}\label{cla6}
{\rm $d_{G}(u)\leq n+k-\omega-1$ for each $u\in V_{2}\backslash V_{2}^{*}$.}
\end{claim}

\begin{proof}
Suppose, to the contrary, that $d_{G}(u)\geq n+k-\omega$ for each $u\in V_{2}\backslash V_{2}^{*}.$
Then $d_{G}(u)+d_{G}(v)\geq (n+k-\omega)+(\omega-1)=n+k-1$ for $u\in V_{2}\backslash V_{2}^{*}$ and $v\in C$.
Note that $G$ is an $(n+k-1)$-closed graph. Then $u$ is adjacent to every vertex of $C,$ and hence $C\cup\{u\}$ is a larger clique, which contradicts Claim \ref{cla5}.
\end{proof}

Notice that $|V_{2}\backslash V_{2}^{*}|=n-|V_{1}|-|V_{2}^{*}\cup V_{3}|=n-\omega-2.$ Hence by Claim \ref{cla6}, we obtain
$$\sum_{u \in V_{2}\backslash V_{2}^{*}}d_{G}(u)\leq (n-\omega-2)(n+k-\omega-1).$$
Then we have
\begin{eqnarray*}
e(G)&\leq& \sum_{u \in V_{1}}d_{G}(u)+\sum_{u \in V_{2}\backslash V_{2}^{*}}d_{G}(u)+e(V_{2}^{*}\cup V_{3}) \\
&\leq& 2(k+1)+(n-\omega-2)(n+k-\omega-1)+{\omega\choose 2}\\
&=&\frac{3}{2}\omega^{2}-(2n+k-\frac{5}{2})\omega+n^{2}+kn-3n+4\\
&\triangleq&f_{2}(\omega).
\end{eqnarray*}
Note that $f_{2}(\omega)$ is a concave function on $\omega.$
If $n-5\leq\omega(G)\leq n-3,$ then
$$e(G)\leq \max\{f_{2}(n-5),f_{2}(n-3)\}={n-3\choose 2}+3k+4<e(G).$$
a contradiction. Therefore, $\omega(G)=n-2$. This completes the proof.
\end{proof}

\begin{rem}
The sufficient condition in terms of edge in Lemma \ref{le11} is best possible. Let $G\cong K_{3}\vee (K_{n-6}+K_{2}+K_{1})$.
Note that $C_{n+1}(G)=G$. Then $G$ is not 2-leaf-connected and $e(G)={n-3\choose 2}+10$.
However, $\omega(G)=n-3.$
\end{rem}

Using the above technical Lemma \ref{le11}, we will present the proof of Theorem \ref{main}.

\medskip
\noindent  \textbf{Proof of Theorem \ref{main}.}
Suppose, to the contrary, that $G$ is not $k$-leaf-connected, where $n\geq k+17, \delta\geq k+1$ and $k\geq2$.
Let $H=C_{n+k-1}(G)$. By Lemma \ref{le111}, $H$ is not $k$-leaf-connected.
Note that $G\subseteq H$. By the assumption $e(G)\geq {n-3\choose 2}+3k+5$, then $e(H)\geq {n-3\choose 2}+3k+5$.
By Lemma \ref{le11}, either $\omega(H)=n-2$ or $H\cong K_{4}\vee(K_{n-7}+3K_{1})$.

Assume that $\omega(H)=n-2$. Next we will characterize the structure of $H$.
Let $C$ be an $(n-2)$-clique of $H$ and $F$ be a subgraph of $H$ induced by $V(H)\backslash C,$ and let $V(F)=\{v_{1}, v_{2}\}$.

\begin{claim}\label{cla7}
{\rm $d_{H}(v_{i})=k+1$ for each $v_{i}\in V(F).$}
\end{claim}

\begin{proof}
Suppose there exists a vertex $v_{i}\in V(F)$ with $d_{H}(v_{i})\geq k+2.$
Then $d_{H}(v_{i})+d_{H}(v)\geq (k+2)+(n-3)=n+k-1$ for any $v\in C$.
Recall that $H=C_{n+k-1}(G).$ Then $v_{i}$ is adjacent to vertex $v$.
Note that $v$ is an arbitrary vertex of $C$. Hence $v_{i}$ is adjacent to all vertices of $C$.
This implies that $\omega(H)\geq n-1,$ a contradiction.
\end{proof}

\begin{claim}\label{cla8}
{\rm $N_{H}(v_{1})\cap C=N_{H}(v_{2})\cap C.$}
\end{claim}

\begin{proof}
Without loss of generality, assume that a vertex $v$ of $C$ is adjacent to $v_{1}$ of $F$, then $d_{H}(v)\geq n-2$.
Therefore, $d_{H}(v)+d_{H}(v_{2})\geq (n-2)+(k+1)=n+k-1.$
Note that $H=C_{n+k-1}(G).$ Then $v$ is also adjacent to vertex $v_{2}$.
Hence $N_{H}(v_{1})\cap C=N_{H}(v_{2})\cap C.$
\end{proof}

Let $|N_{H}(v_{i})\cap C|=t$. Note that $|V(F)|=2$. By Claim \ref{cla7}, we know that $d_{H}(v_{i})=k+1.$
Then $t\geq k.$ On the other hand, $t\leq d_{H}(v_{i})=k+1$. Hence $k\leq t\leq k+1.$
Next, we will discuss the following two cases.

\vspace{1.5mm}
\noindent\textbf{Case 1.} $t=k$.
\vspace{1mm}

Then $H\cong K_{k}\vee (K_{n-k-2}+K_{2})$. Note that $G-V(K_{k})$ is not connected.
Then $G$ has no spanning tree such that $V(K_{k})$ is precisely the set of leaves,
and this implies that $G$ is not $k$-leaf-connected. Note that $e(H)={n-2\choose 2}+2k+1>{n-3\choose 2}+3k+5.$
Hence $H\cong K_{k}\vee (K_{n-k-2}+K_{2}).$

\vspace{1.5mm}
\noindent\textbf{Case 2.} $t=k+1$.
\vspace{1mm}

Then $H\cong K_{k+1}\vee (K_{n-k-3}+2K_{1})$.
By Theorem 1.5 in \cite{Ao2022}, we know that $K_{k+1}\vee (K_{n-k-3}+2K_{1})$ is $k$-leaf-connected for $k\geq 3$, a contradiction.
However, $K_{3}\vee (K_{n-5}+2K_{1})$ is not 2-leaf-connected. Notice that $e(H)={n-2\choose 2}+6>{n-3\choose 2}+11$.
Therefore, $H\cong K_{3}\vee (K_{n-5}+2K_{1}).$

By the above proof, we have $H=C_{n+k-1}(G)\in\{ K_{k}\vee (K_{n-k-2}+K_{2}), K_{3}\vee (K_{n-5}+2K_{1}),
K_{4}\vee (K_{n-7}+3K_{1})\},$
as desired. \hspace*{\fill}$\Box$

\section{Applications}

As applications, we will provide sufficient spectral conditions to guarantee a graph to be $k$-leaf-connected.
The following lemmas are used in the sequel.

\begin{lem}\label{le12}
Let $H\cong K_{k}\vee (K_{n-k-2}+K_{2}).$ \\
(i) If $n\geq 2k+8$, then $\rho(H)>\frac{k}{2}+\sqrt{n^{2}-(k+8)n+\frac{k^{2}}{4}+7k+23}.$\\
(ii) If $n\geq 3k+10$, then $q(H)> 2n-8+\frac{6k+16}{n-1}.$\\
(iii) If $n\geq 3k+9$, then $\rho(\overline{H})< \sqrt{\frac{(n-k)(3n-3k-11)}{n}}.$
\end{lem}

\begin{proof}
(i) Note that $K_{n-2}$ is a proper subgraph of $H.$ Then for $n\geq 2k+8$, we have
$$\rho(H)>\rho(K_{n-2})=n-3> \frac{k}{2}+\sqrt{n^{2}-(k+8)n+\frac{k^{2}}{4}+7k+23}.$$

(ii) For $n\geq 3k+10,$ by a direct calculation, we obtain that
$$q(H)>q(K_{n-2})=2n-6> 2n-8+\frac{6k+16}{n-1}.$$

(iii) Obviously, $\overline{H}\cong kK_{1}\cup[(n-k-2)K_{1}\vee 2K_{1}]$. For $n\geq 3k+9$, we have
$$\rho(\overline{H})=\rho(K_{2,n-k-2})=\sqrt{2(n-k-2)}<\sqrt{\frac{(n-k)(3n-3k-11)}{n}},$$ as desired.
\end{proof}

\begin{lem}\label{le13}
Let $H\cong K_{3}\vee (K_{n-5}+2K_{1}).$ \\
(i) If $n\geq9$, then $\rho(H)> 1+\sqrt{n^{2}-10n+38}.$\\
(ii) If $n\geq10$, then $q(H)> 2n-8+\frac{28}{n-1}.$ \\
(iii) If $n\geq17$, then $\rho(\overline{H})<\sqrt{\frac{(n-2)(3n-17)}{n}}.$
\end{lem}

\begin{proof}
(i) Let $R(A)$ be an equitable quotient matrix of the adjacency matrix $A(H)$ with respect to the partition $(V(K_{3}), V(K_{n-5}), V(2K_{1})).$
In the proof of Theorem 4.2 \cite{Ao2022}, we known that the characteristic polynomial of $R(A)$ is $P_{R(A)}(x)=x^{3}-(n-4)x^{2}-(n+3)x+6n-36,$
and $P_{R(A)}(x)$ is a monotonically increasing function on $[\frac{n-4+\sqrt{n^{2}-5n+25}}{3}, +\infty).$
Note that $\rho(H)=\lambda_{1}(R(A))>\frac{n-4+\sqrt{n^{2}-5n+25}}{3}$ and
$$1+\sqrt{n^{2}-10n+38}>\frac{n-4+\sqrt{n^{2}-5n+25}}{3}.$$
By Maple, $P_{R(A)}(1+\sqrt{n^{2}-10n+38})<0=P_{R(A)}(\rho(H))$
for $n\geq 9$. This implies that $\rho(H)>1+\sqrt{n^{2}-10n+38}$.

(ii) Let $R(Q)$ be an equitable quotient matrix of the signless Laplacian matrix $Q(H)$
with respect to the partition $(V(K_{3}), V(K_{n-5}), V(2K_{1})).$
In the proof of Theorem 4.7 \cite{Ao2022}, the characteristic polynomial of $R(Q)$ is $P_{R(Q)}(x)=x^{3}-(3n-5)x^{2}+(2n^{2}-n-24)x-6n^{2}+42n-72$, and
$P_{R(Q)}(x)$ is a monotonically increasing function on $[\frac{3n-5+\sqrt{3n^{2}-27n+97}}{3}, +\infty).$
Note that $q(H)>\frac{3n-5+\sqrt{3n^{2}-27n+97}}{3}$ and
$$2n-8+\frac{28}{n-1}>\frac{3n-5+\sqrt{3n^{2}-27n+97}}{3}.$$
By a simple calculation, we have
$P_{R(Q)}(2n-8+\frac{28}{n-1})<0=P_{R(Q)}(q(H))$ for $n\geq 10$.
Hence, $q(H)>2n-8+\frac{28}{n-1}$.

(iii) We have $\overline{H}\cong 3K_{1}\cup[(n-5)K_{1}\vee K_{2}].$
Let $RC(A)$ be an equitable quotient matrix of the adjacency matrix $A(\overline{H})$
with respect to the partition $(V(3K_{1}), V((n-5)K_{1}), V(K_{2})).$
One can see that
\begin{equation}
RC(A)=       %开始数学环境
\left(                 %左括号
  \begin{array}{ccc}   %该矩阵一共3列，每一列都居中放置
    0 &  0  & 0\\  %第一行元素
    0 &  0  & 2\\  %第二行元素
    0 & n-5 & 1\\  %第三行元素
  \end{array}\nonumber
\right).                 %右括号
\end{equation}
Then the characteristic polynomial of $RC(A)$ is given by $P_{RC(A)}(x)=x(x^{2}-x-2n+10).$
By a direct calculation,
$\rho(\overline{H})=\frac{1+\sqrt{8n-39}}{2}< \sqrt{\frac{(n-2)(3n-17)}{n}}$
for $n\geq 17$.
\end{proof}

\begin{lem}\label{le14}
Let $H\cong K_{4}\vee (K_{n-7}+3K_{1}).$ \\
(i) If $n\geq 9$, then $\rho(H)< 1+\sqrt{n^{2}-10n+38}.$\\
(ii) If $n\geq 9$, then $q(H)< 2n-8+\frac{28}{n-1}.$\\
(iii) If $n\geq 7$, then $\rho(\overline{H})> \sqrt{\frac{(n-2)(3n-17)}{n}}.$
\end{lem}

\begin{proof}

(i) Let $R(A)$ be an equitable quotient matrix of the adjacency matrix $A(H)$ with respect to the partition $(V(K_{4}), V(K_{n-7}), V(3K_{1})).$
One can see that
\begin{equation}
R(A)=       %开始数学环境
\left(                 %左括号
  \begin{array}{ccc}   %该矩阵一共3列，每一列都居中放置
    3 & n-7 & 3\\  %第一行元素
    4 & n-8 & 0\\  %第二行元素
    4 & 0   & 0\\  %第三行元素
  \end{array}\nonumber
\right).                 %右括号
\end{equation}
Then the characteristic polynomial of $R(A)$ is given by $P_{R(A)}(x)=x^{3}-(n-5)x^{2}-(n+8)x+12n-96.$
By Lemma \ref{le15}, we know that $\rho(H)=\lambda_{1}(R(A))$ is the largest root of the equation $P_{R(A)}(x)=0.$
Let $P_{R(A)}^{'}(x)=3x^{2}-2(n-5)x-n-8=0.$ We can solve this equation to obtain that
$$x_{1}=\frac{n-5-\sqrt{n^{2}-7n+49}}{3}~~ \mbox{and} ~~x_{2}=\frac{n-5+\sqrt{n^{2}-7n+49}}{3}.$$
Then $P_{R(A)}(x)$ is a monotonically increasing function on $[x_{2}, +\infty).$
Note that $\rho(H)=\lambda_{1}(R(A))>x_{2}$ and
$1+\sqrt{n^{2}-10n+38}>x_{2}.$
By Maple,
$P_{R(A)}(1+\sqrt{n^{2}-10n+38})>0=P_{R(A)}(\rho(H))$
for $n\geq 9$. This implies that $\rho(H)<1+\sqrt{n^{2}-10n+38}.$

(ii) Let $R(Q)$ be an equitable quotient matrix of the signless Laplacian matrix $Q(H)$
with respect to the partition $(V(K_{4}), V(K_{n-7}), V(3K_{1})).$ Then
\begin{equation}
R(Q)=       %开始数学环境
\left(                 %左括号
  \begin{array}{ccc}   %该矩阵一共3列，每一列都居中放置
    n+2 & n-7   & 3\\  %第一行元素
    4   & 2n-12 & 0\\  %第二行元素
    4   & 0     & 4\\  %第三行元素
  \end{array}\nonumber
\right).          %右括号
\end{equation}
Then the characteristic polynomial of $R(Q)$ is given by $P_{R(Q)}(x)=x^{3}-3(n-2)x^{2}+(2n^{2}-48)x-8n^{2}+72n-160.$
By Lemma \ref{le15}, we have $q(H)=\lambda_{1}(R(Q))$ is the largest root of the equation $P_{R(Q)}(x)=0.$
Let $P_{R(Q)}^{'}(x)=3x^{2}-6(n-2)x+2n^{2}-48=0.$
The two roots $x_{1}$ and $x_{2}$ of this equation are as follows:
$$x_{1}=\frac{3n-6-\sqrt{3n^{2}-36n+180}}{3}~~ \mbox{and} ~~x_{2}=\frac{3n-6+\sqrt{3n^{2}-36n+180}}{3}.$$
Then $P_{R(Q)}(x)$ is a monotonically increasing function on $[x_{2}, +\infty).$
Note that $q(H)>x_{2}$ and $2n-8+\frac{28}{n-1}>x_{2}$.
By a simple calculation, we have
$P_{R(Q)}(2n-8+\frac{28}{n-1})>0=P_{R(Q)}(q(H))$ for $n\geq 9$.
Hence $q(H)<2n-8+\frac{28}{n-1}.$

(iii) It is easy to see that $\overline{H}\cong 4K_{1}\cup[(n-7)K_{1}\vee K_{3}]$.
Let $RC(A)$ be an equitable quotient matrix of the adjacency matrix $A(\overline{H})$
with respect to the partition $(V(4K_{1}), V((n-7)K_{1}), V(K_{3})).$
One can see that
\begin{equation}
RC(A)=       %开始数学环境
\left(                 %左括号
  \begin{array}{ccc}   %该矩阵一共3列，每一列都居中放置
    0 &  0  & 0\\  %第一行元素
    0 &  0  & 3\\  %第二行元素
    0 & n-7 & 2\\  %第三行元素
  \end{array}\nonumber
\right).                 %右括号
\end{equation}
Then the characteristic polynomial of $RC(A)$ is given by $P_{RC(A)}(x)=x(x^{2}-2x-3n+21).$
By a direct calculation, we have
$\rho(\overline{H})=1+\sqrt{3n-20}> \sqrt{\frac{(n-2)(3n-17)}{n}}$
for $n\geq 7$.
\end{proof}

Ao, Liu, Yuan and Li \cite{Ao2022} presented sufficient conditions to guarantee a graph to be $k$-leaf-connected
in terms of the (signless Laplacian) spectral radius of $G$ or its complement.

\begin{thm}[Ao, Liu, Yuan and Li \cite{Ao2022}]\label{th4}
Let $G$ be a connected graph of order $n$ and minimum degree $\delta \geq k+1$, where $2\leq k\leq n-4.$ Then\\
(i) If $\rho(G)\geq \frac{k}{2}+\sqrt{n^{2}-(k+6)n+\frac{k^{2}}{4}+5k+11},$
then $G$ is $k$-leaf-connected unless $G\in \{K_{3}\vee3K_{1}, K_{4}\vee4K_{1}\}.$\\
(ii) If $q(G)\geq 2n-6+\frac{4k+6}{n-1},$
then $G$ is $k$-leaf-connected unless $G\cong K_{4}\vee 4K_{1}.$\\
(iii) If $\rho(\overline{G})\leq \sqrt{\frac{(n-k)(2n-2k-5)}{n}},$
then $G$ is $k$-leaf-connected.
\end{thm}

In this paper, we improve the above result as follows.

\begin{thm}\label{main2}
Let $G$ be a connected graph of order $n\geq k+17$ and minimum degree $\delta \geq k+1$, where $k\geq 2.$
If one of the following holds,\\
(i) $\rho(G)\geq \frac{k}{2}+\sqrt{n^{2}-(k+8)n+\frac{k^{2}}{4}+7k+23},$\\
(ii) $q(G)\geq 2n-8+\frac{6k+16}{n-1},$\\
(iii) $\rho(\overline{G})\leq \sqrt{\frac{(n-k)(3n-3k-11)}{n}},$\\
then $G$ is $k$-leaf-connected unless $C_{n+k-1}(G)\in\{ K_{k}\vee (K_{n-k-2}+K_{2}), K_{3}\vee (K_{n-5}+2K_{1})\}$.
\end{thm}

\medskip
\noindent\textbf{Proof.}
Suppose, to the contrary, that $G$ is not $k$-leaf-connected.

(i) By Lemma \ref{le8} and Proposition \ref{pro9}, we have
  $$\rho(G)\leq\frac{\delta-1}{2}+\sqrt{2e(G)-\delta n+\frac{(\delta+1)^{2}}{4}}\leq\frac{k}{2}+\sqrt{2e(G)-(k+1)n+\frac{k^{2}}{4}+k+1}.$$
Since $\rho(G)\geq \frac{k}{2}+\sqrt{n^{2}-(k+8)n+\frac{k^{2}}{4}+7k+23}$, we have
$e(G)\geq {n-3\choose 2}+3k+5.$ Let $H=C_{n+k-1}(G).$
By Theorem \ref{main}, we have $H\in \{K_{k}\vee (K_{n-k-2}+K_{2}), K_{3}\vee (K_{n-5}+2K_{1}), K_{4}\vee(K_{n-7}+3K_{1})\}.$
Assume that $H\cong K_{4}\vee(K_{n-7}+3K_{1})$. According to (i) of Lemma \ref{le14}, $\rho(G)\leq\rho(H)<1+\sqrt{n^{2}-10n+38},$ a contradiction.
For $H\in \{K_{k}\vee (K_{n-k-2}+K_{2}), K_{3}\vee (K_{n-5}+2K_{1})$ and $n\geq k+17,$ by (i) of Lemmas \ref{le12} and \ref{le13},
we can not compare completely $\rho(G)$ with $\frac{k}{2}+\sqrt{n^{2}-(k+8)n+\frac{k^{2}}{4}+7k+23}.$
For the brevity of discussion, we have $C_{n+k-1}(G)=H\in\{ K_{k}\vee (K_{n-k-2}+K_{2}), K_{3}\vee (K_{n-5}+2K_{1})\}.$

(ii) By Lemma \ref{le10}, we have $q(G)\leq \frac{2e(G)}{n-1}+n-2.$
Note that $q(G)\geq 2n-8+\frac{6k+16}{n-1}.$
Then $e(G)\geq {n-3\choose 2}+3k+5.$ Let $H=C_{n+k-1}(G).$ By Theorem \ref{main}, we have $H\in \{K_{k}\vee (K_{n-k-2}+K_{2}), K_{3}\vee (K_{n-5}+2K_{1}), K_{4}\vee(K_{n-7}+3K_{1})\}.$
Suppose that $H\cong K_{4}\vee(K_{n-7}+3K_{1})$. By (ii) of Lemma \ref{le14},
$q(G)\leq q(H)<2n-8+\frac{28}{n-1},$ a contradiction.
Therefore, $C_{n+k-1}(G)=H\in\{ K_{k}\vee (K_{n-k-2}+K_{2}), K_{3}\vee (K_{n-5}+2K_{1})\}.$

(iii) Let $H=C_{n+k-1}(G)$. Similar to the proof of Theorem 4.4 in \cite{Ao2022}, we can obtain that
$$\rho(\overline{H})\geq \sqrt{\frac{(n-k)e(\overline{H})}{n}}.$$
Note that $\overline{H}\subseteq \overline{G}.$
Then we have
  $$\rho(\overline{H})\leq \rho(\overline{G})\leq \sqrt{\frac{(n-k)(3n-3k-11)}{n}},$$
and therefore, $$\sqrt{\frac{(n-k)e(\overline{H})}{n}}\leq \rho(\overline{H})\leq \rho(\overline{G})\leq \sqrt{\frac{(n-k)(3n-3k-11)}{n}}.$$
It is easy to check that $e(\overline{H})\leq 3n-3k-11$ and
  $$e(H)={n\choose 2}-e(\overline{H})\geq {n-3\choose 2}+3k+5. $$
Applying Theorem \ref{main} on $H$, we have $C_{n+k-1}(H)=H\in \{K_{k}\vee (K_{n-k-2}+K_{2}), K_{3}\vee (K_{n-5}+2K_{1}), K_{4}\vee(K_{n-7}+3K_{1})\}.$
Assume that $H\cong K_{4}\vee(K_{n-7}+3K_{1})$. By (iii) of Lemma \ref{le14}, $\rho(\overline{G})\geq\rho(\overline{H})>\sqrt{\frac{(n-2)(3n-17)}{n}},$
a contradiction. Hence $C_{n+k-1}(G)=H\in\{ K_{k}\vee (K_{n-k-2}+K_{2}), K_{3}\vee (K_{n-5}+2K_{1})\}.$
This completes the proof of Theorem \ref{main2}.  \hspace*{\fill}$\Box$

\section*{Declaration of competing interest}
The authors declare that they have no conflict of interest.

\section*{Acknowledgement}
We would like to thank the anonymous referees for their valuable suggestions that improved the presentation of this paper.

\end{document}